\documentclass[12pt]{amsart}
\usepackage[a4paper, top=3cm, bottom=3cm, left=2.5cm, right=2.5cm]{geometry}
\usepackage[utf8]{inputenc}
\usepackage[T1]{fontenc}
\usepackage{amsmath}
\usepackage{amsfonts}
\usepackage{amssymb}
\usepackage{amsthm}
\usepackage{braket}
\usepackage{hyperref}
\usepackage[dvipsnames]{xcolor}
\usepackage{tikz}
\usepackage[backend=bibtex,style=numeric,firstinits=true, maxbibnames=99]{biblatex}

\bibliography{References}

\DeclareMathOperator{\vol}{vol}
\DeclareMathOperator{\Tr}{Tr}

\renewcommand{\epsilon}{\varepsilon}

\newtheorem{theorem}{Theorem}
\newtheorem{proposition}[theorem]{Proposition}

\theoremstyle{definition}

\newcommand{\normiii}[1]{{\left\vert\kern-0.25ex\left\vert\kern-0.25ex\left\vert #1 
    \right\vert\kern-0.25ex\right\vert\kern-0.25ex\right\vert}}

    \def\R {\mathbb{R}}

       \newcommand{\N}{\mathbb{N}}

       \usepackage{graphicx}
\usepackage{caption}
\usepackage{subcaption}

\title[The critical case for the concentration of eigenfunctions]{The critical case for the concentration of eigenfunctions on singular Riemannian manifolds}
\author{Charlotte Dietze}
\address[Charlotte Dietze]{Sorbonne Université, CNRS, Laboratoire Jacques-Louis Lions,
4 place Jussieu,
75005 Paris,
France
}
 \email{Charlotte.Dietze@sorbonne-universite.fr}
 
 	\subjclass[2010]{Primary: 58C40; Secondary: 53C17, 35P20}
\keywords{Grushin, sub-Riemannian geometry, Weyl's law, singular Riemannian metric, gas giants}

\begin{document}
\maketitle
\begin{abstract}
We consider a compact Riemannian manifold with boundary with a certain class of critical singular Riemannian metrics that are singular at the boundary. The corresponding Laplace-Beltrami operator can be seen as a Grushin-type operator plus a potential. 
We show in the critical case that the average density of eigenfunctions for the Laplace-Beltrami operator with eigenvalues below $\lambda>0$ is distributed over all length scales between $\lambda^{-1/2}$ and $1$ near the boundary. We give a precise description of this distribution as $\lambda\to\infty$.
\end{abstract}

\section{Introduction}
Let $X$ be an $(n+1)$-dimensional compact smooth Riemannian manifold with boundary, $n\in\N$. We consider singular Riemannian metrics $g$ on $X$ that are smooth and non-degenerate in the interior of $X$ and have a certain critical singularity near the boundary.

\medskip 

For a suitable choice of coordinates near the boundary, where $x$ is a transverse coordinate to the boundary, we can identify the manifold $X$ near its boundary with $[0,1]\times M$, where $M$ is a smooth manifold of dimension $n$ that corresponds to the boundary manifold. Here $\{0\}\times M$ is identified with the boundary $\partial X$. We will also use, for any $\epsilon>0$, the notation $X_\epsilon$ for the subset of $X$ that is identified with $[0,\epsilon]\times M$.

\subsection*{Definition of the singular Riemannian metric $g$}
For any family $\{g_1(x)\}_{x\in [0,1]}$ of smooth non-degenerate Riemannian metrics $g_1(x)$ on $M$ that depend continuously on the parameter $x\in [0,1]$, we consider singular Riemannian metrics $g$ on $X$ that are smooth and non-singular inside $X$, and are of the form
\begin{equation}\label{eq:g_def}
g= \mathrm{d}x^2+x^{-\beta}g_1(x) \quad \text{on } X_1\cong [0,1]\times M\,.
\end{equation}
Here $\beta\in(0,\infty)$ is a parameter modelling the singularity of $g$ near the boundary. In this paper, we will focus on the case of a critical singularity, that is, $\beta=2/n$.

\medskip

We denote the Laplace-Beltrami operator with respect to $g$ on $X$ with Dirichlet boundary conditions by $\Delta_g$. We use the sign convention that $\Delta_g$ is a non-negative operator.

\subsection*{$\Delta_g$ as a $\beta$-Grushin-type operator plus a potential}
The operator $\Delta_g$ is a model example for an operator in sub-Riemannian geometry. More precisely, $\Delta_g$ can be seen as a $\beta$-Grushin-type operator plus a potential: For $\epsilon_0>0$ small, the metric $g$ is quasi-isometric with a constant close to $1$ depending on $\epsilon_0$ to a metric $\tilde g $ on $X$, where $\tilde g$ satisfies
\begin{equation}\label{eq:g_def}
\tilde g= \mathrm{d}x^2+x^{-\beta}g_1(0) \quad \text{on } X_{\epsilon_0}\cong [0,\epsilon_0]\times M\,.
\end{equation}
On $X_{\epsilon_0}$, we have
\begin{equation}\label{eq:LB_near_boundary_in_coord}
\Delta_{\tilde g}=-\partial_x^2+\frac{C_\beta}{x^2}+x^\beta\Delta_{g_1(0)}\, \text{ with } C_\beta=\frac{\beta n}{4}\left(\frac{\beta n}{4}+1\right)\,,
\end{equation}
where $\Delta_{g_1(0)}$ denotes the Laplace-Beltrami operator with respect to $g_1(0)$ on $M$, see \cite{yves}. The case $\beta=2$ corresponds to the classical Grushin operator 
\begin{equation}
\left(\partial_x\right)^2+\sum_{k=1}^n\left(x\partial_{y_k}\right)^2
\end{equation}
plus a potential $-\frac{C_\beta}{x^2}$. Thus, $-\Delta_g$ can be seen as a $\beta$-Grushin operator plus a potential, see also \cite[(1.9)]{boscain_laurent}, where the case $n=1$, $\beta=2$ is considered. 

\subsection*{Spectral properties of $\Delta_g$ in the literature}


The spectral properties of $\Delta_g$, and more generally Laplace-Beltrami operators of singular Riemannian metrics or operators from sub-Riemannian geometry, have been considered in many works in the literature, see for example \cite{cheeger}. In \cite{menikoff1978eigenvalues}, the authors prove Weyl asymptotics of classical pseudo-differential operators on smooth manifolds with a principal symbol vanishing exactly to second order on a smooth symplectic submanifold, which corresponds to the Grushin case. Their asymptotics also include the critical case, where a logarithmic term appears, compare with \eqref{eq:weyl_crit} below. In \cite{boscain}, the Weyl asymptotics for the Grushin cylinder were computed using an explicit computation. We would also like to mention \cite{rivera2022weighted} where weighted Weyl laws are derived, and logarithmic terms appear. For the Grushin case, we also refer to \cite{abatangelo2025solutions, lamberti2021shape}. In \cite{chitour}, the authors study a general class of singular metrics, inspired by the Grushin model.

\medskip 

For the study of the small time asymptotics of sub-Riemannian heat kernels, we refer to \cite{colin2018spectral,colin2021small,colin2022spectral,chang2015heat}. To treat singular boundaries, Melrose developed the b-calculus \cite{melrose1981transformation}, see also \cite{grieser2001basics} for an introduction. Concerning the localisation of eigenfunctions, in \cite{fahs2024boundary}, the authors consider the magnetic Robin Laplacian and show exponential localisation of eigenfunctions near the boundary. 


\medskip

In the following, we give a detailed overview of the results in \cite{yves,larry} as they motivate and provide the context for the main result in the present paper. 

\medskip

For any $\lambda>0$, we denote by $N(\lambda)$ the number of eigenvalues of $\Delta_g$ below $\lambda$. In \cite[Theorem 1]{yves}, the authors determined the leading-order Weyl asymptotics for $N(\lambda)$. As was pointed out after the completion of \cite{yves}, equivalent results can be found much earlier in the literature in \cite{metivier, vulis}, see also \cite{solomesc}. 

\medskip 

For $\beta<2/n$, which is referred to as the subcritical case, the leading-order Weyl asymptotics of $N(\lambda)$ agree with the asymptotics one would expect if $g$ was a  non-singular Riemannian metric on $X$:
\begin{equation}
N(\lambda)=C^W_{n+1}\vol(X,g)\lambda^{\frac{n+1}{2}}+o\left(\lambda^{\frac{n+1}{2}}\right) \quad \text{ as } \lambda\to\infty.
\end{equation}
Here $C^W_{n+1}>0$ is the classical constant in Weyl's law that only depends on the dimension $n+1$. 

\medskip

In the critical case $\beta=2/n$, \cite[Theorem 1]{yves} states that 
\begin{equation}\label{eq:weyl_crit}
N(\lambda)=C_{n}\vol(M,g_1(0))\lambda^{\frac{n+1}{2}}\log(\lambda)+o\left(\lambda^{\frac{n+1}{2}}\log(\lambda)\right) \quad \text{ as } \lambda\to\infty.
\end{equation}
Here, $C_{n}>0$ is an explicit constant only depending on $n$.  In a slightly different, but in a way more general setting, \eqref{eq:weyl_crit} was also shown in \cite[Theorem 5.1]{chitour}.

\medskip

In the supercritical case $\beta>2/n$, using the notation $d:=n(1+\beta/2)$, we have
\begin{equation}
N(\lambda)=C_{n,\beta}\vol(M,g_1(0))\lambda^{\frac{d}{2}}+o\left(\lambda^{\frac{d}{2}}\right) \quad \text{ as } \lambda\to\infty,
\end{equation}
where $C_{n,\beta}>0$ is a constant only depending on $n$ and $\beta$. 

\medskip

Note that both in the critical and supercritical case, the leading-order term in the Weyl asymptotics involves $\vol(M,g_1(0))$. Put differently, only the behaviour of the metric $g$ near the boundary matters for the leading-order behaviour. This suggests an interesting behaviour near the boundary.

\medskip 

Indeed, in \cite[Theorem 2]{yves}, it is shown that in the critical and supercritical case, the average density of eigenfunctions accumulates at the boundary, see also \cite[Theorem 6.1]{chitour}. Denote by $\{\Phi_j\}_{j\in\N}$ an orthonormal basis in $L^2(X,\vol_g)$ of eigenfunctions of $\Delta_g$ with eigenvalues $\{\lambda_j\}_{j\in\N}$. More precisely, \cite[Theorem 2]{yves} states that the average density of eigenfunctions with eigenvalues below $\lambda$
\begin{equation}\label{eq:average_dens_ef}
\frac{1}{N(\lambda)}\sum_{\lambda_j<\lambda}\left|\Phi_j\right|^2\, \mathrm{d}\vol_g
\end{equation}
converges weakly to the uniform distribution on the boundary $\partial X\cong \{0\}\times M$ with respect to $g_1(0)$  as $\lambda\to\infty$. 

\bigskip

It is a natural question to ask if more can be said about the average density of eigenfunctions defined in \eqref{eq:average_dens_ef}. In particular, one might ask at which scale near the boundary this average density typically lives. For the supercritical case $\beta>2/n$, this was answered in \cite[Theorem 1]{larry}, namely the average density of eigenfunctions with eigenvalues less than $\lambda$ is located at a length-scale $\lambda^{-1/2}$ near the boundary $\partial X$. Furthermore, when zooming in at that scale and letting $\lambda\to\infty$, it converges to an explicitly known profile $B$ . More precisely, identifying $X_1\cong [0,1]\times M$, according to \cite[Theorem 1]{larry}, we have for $\beta>2/n$ for any continuous and bounded test function $f:[0,\infty)\times M\to\R$
\begin{align}\label{eq:th_larry}
\begin{split}
&\qquad\lim_{\lambda\to\infty}\frac{1}{N(\lambda)}\sum_{\lambda_j<\lambda}\int_{[0,1]\times M}f\left(\sqrt{\lambda}x,y\right)\left|\Phi_j(x,y)\right|^2\, \,\mathrm{d}\vol_g(x,y)\\
&=\int_{[0,\infty)\times M}f(z,y)B(z,y)\,\mathrm{d}z\frac{\mathrm{d}\vol_{g_1(0)}(y)}{\vol_{g_1(0)}(M)}\,.
\end{split}
\end{align}

\subsection*{Main result}
An open question that remained was at which scale the average density of eigenfunctions accumulates at the boundary of $X$ in the critical case $\beta=2/n$. The answer to this question is given in Theorem \ref{th:main} below. It states that there is no such scale at which the eigenfunctions concentrate near the boundary. The average density of eigenfunctions is distributed over all length scales between $\lambda^{-1/2}$ and $1$: For all $\gamma\in\left[-1 / 2, 0\right]$, the integral of the average density of eigenfunctions with eigenvalues below $\lambda$ integrated over $X_{\lambda^\gamma}$ converges to $2(1/2+\gamma)$ as $\lambda\to\infty$.


\begin{theorem}\label{th:main}
Let $\gamma\in\left[-1 / 2, 0\right]$. Then
\begin{equation}\label{eq:th_main}
\lim_{\lambda\to\infty}\frac{1}{N(\lambda)}\sum_{\lambda_j<\lambda}\int_{X_{\lambda^\gamma}}\left|\Phi_j\right|^2\, \mathrm{d}\vol_g=2\left(\frac{1}{2}+\gamma\right).
\end{equation}
\end{theorem}
Theorem \ref{th:main} states that for any $\gamma\in\left[-1 / 2, 0\right]$, the proportion of eigenfunctions with eigenvalue less than $\lambda$ located in the neighbourhood of the boundary of $X$ of size $\lambda^\gamma$ is approximately $2\left(1/2+\gamma\right)$ for $\lambda$ large. 

\medskip 

In particular, for $\gamma =-1 / 2$, this proportion is equal to zero. This is in sharp contrast to the corresponding result in the supercritical case \cite[Theorem 1]{larry}, where ``most of the eigenfunctions live on a scale $\lambda^{-1/2}$ near the boundary''.

\medskip

More precisely, due to the linear behaviour in $\gamma\in\left[-1 / 2, 0\right]$ of the right-hand side in \eqref{eq:th_main}, we can also say that the eigenfunctions are localised at scale $\lambda^\gamma$, uniformly distributed in $\gamma\in\left[-1 / 2, 0\right]$ for $\lambda\to\infty$.

\medskip

In order to facilitate the comparison with \cite[Theorem 1]{larry}, see also \eqref{eq:th_larry} above, we present another version of Theorem \ref{th:main} involving a test function $f$, where we also take the distribution along the boundary variable into account:
\begin{theorem}\label{th:test_fct}
Let $f:(-\infty,0]\times M\to\R$ be continuous and bounded. Then
\begin{align}\label{eq:th_test_fct}
\begin{split}
&\qquad\lim_{\lambda\to\infty}\frac{1}{N(\lambda)}\sum_{\lambda_j<\lambda}\int_{[0,1]\times M}f\left(\frac{\log(x)}{\log(\lambda)},y\right)\left|\Phi_j(x,y)\right|^2\, \,\mathrm{d}\vol_g(x,y)\\
&=2\int_{[-1/2,0]\times M}f(\tilde{\gamma},y)\,\mathrm{d}\tilde{\gamma}\frac{\mathrm{d}\vol_{g_1(0)}(y)}{\vol_{g_1(0)}(M)}\,.
\end{split}
\end{align}
\end{theorem}

\subsection*{Structure of the paper}
In Section \ref{s:prop}, we prove the key ingredient of the proof of Theorem \ref{th:main}, which are some trace asymptotics, detailed in Proposition \ref{prop:n_tr} below. Then, in Section \ref{s:th_proofs}, we prove Theorem \ref{th:main} using Proposition \ref{prop:n_tr}. We also explain the proof of Theorem \ref{th:test_fct}.

\subsection*{Notation}
For any self-adjoint operator $H$, we denote by $\Tr(H)_-$ the sum of negative eigenvalues of $H$ if this quantity is finite. Otherwise, we set $\Tr(H)_-=-\infty$. In particular, we use the sign convention that we always have $\Tr(H)_-\leqslant 0$.

Furthermore, we denote by $N(H, \lambda)$ the number of eigenvalues of $H$ that are below $\lambda\in\R$ counted with multiplicity.

\subsection*{Acknowledgments}
I would like to thank Yves Colin de Verdière and Emmanuel Trélat for introducing me to the topic of sub-Riemannian Laplacians, for many very helpful discussions, and for remarks on the literature. I would also like to thank Larry Read for helpful discussions and remarks. I would like to thank Phan Thành Nam and Laure Saint-Raymond for their support and for making possible several visits to Institut des Hautes Etudes Scientifiques, where I started working on this topic. I acknowledge the support from the European
Research Council via the ERC CoG RAMBAS (Project-Nr. 10104424).

\section{Trace asymptotics}\label{s:prop}
In this section, we prove the key ingredient of the proof of Theorem \ref{th:main}, which is the following:
\begin{proposition}\label{prop:n_tr}
Let $\gamma \in\left[-1 / 2, 0\right]$. Then, uniformly in $\delta \in[-1 / 2,1 / 2]$,
\begin{align}\label{eq:goal}
\begin{split}
&\qquad \frac{1}{\lambda N\left(\lambda\right)}\left(\left|\Tr\left(\Delta_g- \lambda\right)_-\right|-\left|\Tr\left(\Delta_g+\delta \lambda 1_{X_{\lambda^{\gamma}}}- \lambda\right)_-\right|\right)\\
&=\frac{2}{n+3}\left(1-(1-\delta)^{\frac{n+3}{2}}\right) 2(1 / 2+\gamma)+o(1) \quad \text{ as } \lambda \rightarrow \infty.
\end{split}
\end{align}
Here $1_{X_{\lambda^{\gamma}}}$ denotes the multiplication operator with the indicator function of $X_{\lambda^{\gamma}}$.
\end{proposition}
\begin{proof}
As we have seen in \cite{yves, drarbeit}, the problem can be reduced to a simpler one involving the one-dimensional operator
\begin{equation}
P_{\mu}:=-\partial_{x}^{2}+\frac{3}{4x^{2}}+\mu x^{\beta} \quad \text { on }\left[0, \varepsilon_{0}\right] \text {, }
\end{equation}
where 
$\varepsilon_{0}>0$ is small, but fixed. Compare also with \eqref{eq:LB_near_boundary_in_coord}.
We put Dirichlet boundary conditions at $x=0$ and Dirichlet or Neumann boundary conditions at $x=\varepsilon_{0}$. As in \cite[Section 3.13]{drarbeit}, the leading order asymptotics of 
\begin{equation}
\left|\Tr\left(\Delta_g- \lambda\right)_-\right|-\left|\Tr\left(\Delta_g+\delta \lambda 1_{X_{\lambda^{\gamma}}}- \lambda\right)_-\right|\,,
\end{equation}
up to an explicit constant, is given by the leading order asymptotics of 
\begin{equation}\label{eq:davon_asy_ausr}
\sum_{j=\lambda^{n / 2}}^{E \lambda^{(n+1)/2}}\left(
\left|\Tr_{\left[0, \varepsilon_{0}\right]}^{D, D / N}\left(P_{j^{2/n}}- \lambda\right)_{-}\right|-\left|\Tr_{\left[0, \varepsilon_{0}\right]}^{D, D/N}\left(P_{j^{2/n}}+\delta \lambda 1_{\left[0, \lambda^{\gamma}\right]}- \lambda\right)_-\right|
\right)
\end{equation}
for $E>0$ fixed, but chosen arbitrarily small. Later, after letting $\lambda \rightarrow \infty$, we will let $E \rightarrow 0$. Here the superscript $D, D / N$ denotes the Dirichlet boundary conditions at $x=0$ and the Dirichlet or Neumann boundary conditions at $x=\varepsilon_{0}$. 

\medskip

Similarly, the leading order asymptotics of $N\left(\lambda\right)$ is given, up to the same explicit constant, by
\begin{equation}
\sum_{j=\lambda^{n / 2}}^{E \lambda^{(n+1)/2}} N\left(P_{j^{2/n}}, \lambda\right).
\end{equation}
By \cite[Section 3.13]{drarbeit}, we have
\begin{align}\label{eq:asympt_n}
\begin{split}
\lim_{E\to 0}\lim_{\lambda\to\infty}\left(\lambda^{\frac{n+1}{2}} \log (\lambda)\right)^{-1}\sum_{j=\lambda^{n / 2}}^{E \lambda^{(n+1)/2}}
N\left(P_{j^{2/n}}, \lambda\right)= \frac{1}{2}A
,
\end{split}
\end{align}
where $A>0$ was defined in \cite[(3.32)]{drarbeit} by
\begin{equation}\label{eq:A_def}
A:=\frac{1}{\pi}\int_0^1\sqrt{1-z^\beta}\, dz.
\end{equation}

\bigskip

Let us now turn to the computation of the leading-order asymptotics of \eqref{eq:davon_asy_ausr}. First note that by Dirichlet-Neumann bracketing, we have for any $\delta \in[-1 / 2,1 / 2]$ (and in particular also for $\delta=0$)
\begin{align}\label{eq:dsplitgamma}
\begin{split}
\left|\Tr_{\left[0, \varepsilon_{0}\right]}^{D, D / N}\left(P_{j^{2/n}}+\delta \lambda 1_{\left[0, \lambda^{\gamma}\right]}- \lambda\right)_-\right|&
\geqslant \left|\Tr_{\left[0, \lambda^{\gamma}\right]}^{D, D}\left(P_{j^{2/n}}- (1-\delta)\lambda\right)_-\right|+\left|\Tr_{\left[\lambda^{\gamma}, \varepsilon_{0}\right]}^{D, D / N}\left(P_{j^{2/n}}- \lambda\right)_-\right|
\end{split}
\end{align}
and
\begin{align}\label{eq:nsplitgamma}
\begin{split}
\left|\Tr_{\left[0, \varepsilon_{0}\right]}^{D, D / N}\left(P_{j^{2/n}}+\delta \lambda 1_{\left[0, \lambda^{\gamma}\right]}- \lambda\right)_-\right|&
\leqslant \left|\Tr_{\left[0, \lambda^{\gamma}\right]}^{D, N}\left(P_{j^{2/n}}- (1-\delta)\lambda\right)_-\right|+\left|\Tr_{\left[\lambda^{\gamma}, \varepsilon_{0}\right]}^{N, D / N}\left(P_{j^{2/n}}- \lambda\right)_-\right|
\end{split}
\end{align}
As in \cite[Section 3.13]{drarbeit}, in order to compute the leading order asymptotics of the terms on the right-hand side of \eqref{eq:dsplitgamma} and \eqref{eq:nsplitgamma},
one can decompose the interval $\left[0, \lambda^{\gamma}\right]$ or $\left[\lambda^{\gamma}, \varepsilon_{0}\right]$ into smaller subintervals using Dirichlet-Neumann bracketing. In particular, the Dirichlet or Neumann boundary conditions at $x=\lambda^{\gamma}$ do not make a difference for the leading-order asymptotics. 

\bigskip

Using \eqref{eq:dsplitgamma} and \eqref{eq:nsplitgamma} for $\delta$ and $\delta=0$, in order to compute the leading order asymptotics of \eqref{eq:davon_asy_ausr}, it suffices to compute the leading-order asymptotics of 
\begin{equation}\label{eq:asy_einf}
\sum_{j=\lambda^{n / 2}}^{E \lambda^{(n+1)/2}}\left(
\left|\Tr_{\left[0, \lambda^{\gamma}\right]}^{D, D/N}\left(P_{j^{2/n}}- \lambda\right)_-\right|
- \left|\Tr_{\left[0, \lambda^{\gamma}\right]}^{D, D/N}\left(P_{j^{2/n}}- (1-\delta)\lambda\right)_-\right|
\right).
\end{equation}

\bigskip

 Recall from \cite[Proposition 5]{yves} that for any $\mu, a>0$, the operator $P_{\mu}$ on $[0, a]$ is unitarily equivalent to
\begin{equation}
\mu^{2/(2+\beta)} P_{1} \quad \text { on }\left[0, \mu^{1/(2+\beta)} a\right] \text {. }
\end{equation}
Using $\beta=2/n$ and taking $\mu=j^{2/n}$, we obtain
\begin{equation}\label{eq:nach_unit}
\left|\Tr_{\left[0, \lambda^{\gamma}\right]}^{D, D / N}\left(P_{j^{2/n}}-(1-\delta) \lambda\right)_-\right|
=j^{\frac{2}{n+1}}\left|\Tr_{\left[0, j^{1 /(n+1)} \lambda^{\gamma}\right]}^{D, D / N}\left(P_{1}-(1-\delta) \lambda{j}^{-\frac{2}{n+1}}\right)_-\right| .
\end{equation}

Note that if $j \geqslant 2^{1 / \beta} \lambda^{1 / \beta-\gamma}$, then for all $\delta \in[-1 / 2,1 / 2]$,
\begin{equation}\label{eq:j_gamma_ineq}
\left(j^{\frac{1}{n+1}} \lambda^{\gamma}\right)^{\beta} \geqslant(1-\delta) \lambda j^{-\frac{2}{n+1}}.
\end{equation}
We use the notation
\begin{equation}\label{eq:omega_delta_def}
\omega_{\delta}:=(1-\delta) \lambda j^{-\frac{2}{n+1}}.
\end{equation}
For all $j \geqslant 2^{1 / \beta} \lambda^{1 / \beta-\gamma}$, we know by \eqref{eq:j_gamma_ineq} that $\left[0, j^{1 /(n+1)} \lambda^{\gamma}\right] \supset\left[0, \omega_{\delta}^{1 / \beta}\right]$ and therefore, 
\begin{equation}\label{eq:d_omega_delta_ung}
\left|\Tr_{\left[0, j^{1 /(n+1)} \lambda^{\gamma}\right]}^{D, D / N}\left(P_{1}-\omega_{\delta}\right)_-\right|\geqslant\left|\Tr_{\left[0, \omega_{\delta}^{1 / \beta}\right]}^{D, D }\left(P_{1}-\omega_{\delta}\right)_-\right|.
\end{equation}
Since $P_{1} \geqslant \omega_{\delta}$ on $\left[\omega_{\delta}^{1 / \beta}, \infty\right)$, we also have
\begin{align}\label{eq:n_omega_delta_ung}
\begin{split}
\left|\Tr_{\left[0, j^{1 /(n+1)} \lambda^{\gamma}\right]}^{D, D / N}\left(P_{1}-\omega_{\delta}\right)_-\right|&\leqslant\left|\Tr_{\left[0, \omega_{\delta}^{1 / \beta}\right]}^{D,  N}\left(P_{1}-\omega_{\delta}\right)_-\right|+\left|\Tr_{\left[\omega_{\delta}^{1 / \beta}, j^{1 /(n+1)} \lambda^{\gamma}\right]}^{N, D/N }\left(P_{1}-\omega_{\delta}\right)_-\right|\\
&=\left|\Tr_{\left[0, \omega_{\delta}^{1 / \beta}\right]}^{D,  N}\left(P_{1}-\omega_{\delta}\right)_-\right|.
\end{split}
\end{align}
Thus, in view of \eqref{eq:nach_unit}, we would like to compute the leading-order asymptotics of 
\begin{equation}\label{eq:davon_wollen_wir_asy}
\sum_{j=\lambda^{n / 2}}^{E \lambda^{(n+1)/2}}j^{\frac{2}{n+1}}\left(
\left|\Tr_{\left[0, \omega_{0}^{1 / \beta}\right]}^{D,  D/N}\left(P_{1}-\omega_{0}\right)_-\right|
- \left|\Tr_{\left[0, \omega_{\delta}^{1 / \beta}\right]}^{D,  D/N}\left(P_{1}-\omega_{\delta}\right)_-\right|
\right).
\end{equation}

\bigskip

Uniformly in $\delta \in[-1 / 2,1 / 2]$, recall from \cite[(3.31)]{drarbeit} that  for every $\eta>0$, there exists $E=E(\eta)>0$ small enough such that for every $j \leqslant E \lambda^{\frac{n+1}{2}}$, we have 
\begin{equation}\label{eq:NintP1}
N_{\left[0, \omega_{\delta}^{1 / \beta}\right]}^{D, D / N}\left(P_{1}, \omega_{\delta}\right) \in\left[(1+\eta)^{-1} A \omega_{\delta}^{\frac{1}{2}+\frac{1}{\beta}},(1+\eta) A \omega_{\delta}^{\frac{1}{2}+\frac{1}{\beta}}\right]\, ,
\end{equation}
where $A>0$ was defined in \eqref{eq:A_def}.
Similarly, for the traces, we have
\begin{equation}\label{eq:TrintP1}
\left|\Tr_{\left[0, \omega_{\delta}^{1 / \beta}\right]}^{D,  D/N}\left(P_{1}-\omega_{\delta}\right)_-\right| \in\left[(1+\eta)^{-1} \frac{2}{n+3}
A \omega_{\delta}^{\frac{3}{2}+\frac{1}{\beta}},(1+\eta)  \frac{2}{n+3} A \omega_{\delta}^{\frac{3}{2}+\frac{1}{\beta}}\right].
\end{equation}
Here we used that
\begin{equation}
\frac{2}{3\pi}\int_0^1\left(1-z^\beta\right)^{\frac{3}{2}}\, dz=\frac{2}{n+3} A.
\end{equation}

\bigskip

Note that
\begin{align}\label{eq:omega_comp}
\begin{split}
&\quad \sum_{j=2^{1 / \beta} \lambda^{1 / \beta-\gamma}}^{E\lambda^{(n+1)/2}}j^{\frac{2}{n+1}}\left(\omega_{\delta}^{\frac{n+3}{2}}-\omega_{0}^{\frac{n+3}{2}}\right) 
=  \left(1-(1-\delta)^{\frac{n+3}{2}}\right)\lambda^{\frac{n+3}{2}} \sum_{j=2^{1 / \beta} \lambda^{1 / \beta-\gamma}}^{E \lambda^{{(n+1)/2}} }j^{-1} \\
= & \left(1-(1-\delta)^{\frac{n+3}{2}}\right) \lambda^{\frac{n+3}{2}}\left(\log \left(E \lambda^{\frac{n+1}{2}}\right)-\log \left(2^{-1 / \beta} \lambda^{1 / \beta-\gamma}\right)\right)+o\left(\lambda^{\frac{n+3}{2}} \log (\lambda)\right) \\
= & \left((1-\delta)^{\frac{n+3}{2}}-1\right)\left(\frac{1}{2}+\gamma\right) \lambda^{\frac{n+3}{2}} \log (\lambda)+o\left(\lambda^{\frac{n+3}{2}} \log (\lambda)\right)  
\end{split}
\end{align}
as $\lambda\to\infty$. Here we used that $\beta=2/n$ in the last step. 

\bigskip

Furthermore, the sum of the terms corresponding to $\lambda^{n / 2} \leqslant j \leqslant 2^{1 / \beta} \lambda^{1 / \beta-\gamma}$ is of subleading order. Indeed, for all $j \in \mathbb{N}$ and all $\delta \in[-1 / 2,1 / 2]$,
\begin{equation}\label{eq:trlowerordertermest}
\left|\Tr_{\left[0, \lambda^{\gamma}\right]}^{D, D / N}\left(P_{j^{2/n}}-(1-\delta) \lambda\right)_-\right| \leqslant \left|\Tr_{\left[0, \lambda^{\gamma}\right]}^{D, D / N}\left(-\partial_{x}^{2}-(1-\delta) \lambda\right)_-\right|\leqslant 1+\frac{2}{3\pi} \lambda^{\gamma} \left((1-\delta) \lambda\right)^{\frac{3}{2}.}
\end{equation}
By \eqref{eq:trlowerordertermest}, we obtain for all $\lambda \geqslant 1$
\begin{equation}\label{eq:lower_order}
\sum_{j=\lambda^{n / 2}}^{2^{1 / \beta} \lambda^{1 / \beta-\gamma}} \quad \left|\Tr_{\left[0, \lambda^{\gamma}\right]}^{D, D / N}\left(P_{j^{2/n}}-(1-\delta) \lambda\right)_-\right| \leqslant C \lambda^{\frac{1}{\beta}-\gamma} \lambda^{\gamma+\frac{3}{2}}=C \lambda^{\frac{n+3}{2}}
\end{equation}
where the constant $C>0$ is uniform in $\lambda \geqslant 1$ and $\delta \in[-1 / 2,1 / 2]$. 

\bigskip

Combining \eqref{eq:TrintP1}, \eqref{eq:omega_comp} and \eqref{eq:lower_order}, we obtain
\begin{align}\label{eq:asympt_trace}
\begin{split}
&\lim_{E\to 0}\lim_{\lambda\to\infty}\left(\lambda^{\frac{n+3}{2}} \log (\lambda)\right)^{-1} \sum_{j=\lambda^{n / 2}}^{E \lambda^{(n+1)/2}}j^{\frac{2}{n+1}}
\left(
\left|\Tr_{\left[0, \omega_{0}^{1 / \beta}\right]}^{D,  D/N}\left(P_{1}-\omega_{0}\right)_-\right|
- \left|\Tr_{\left[0, \omega_{\delta}^{1 / \beta}\right]}^{D,  D/N}\left(P_{1}-\omega_{\delta}\right)_-\right|
\right)\\
&= \left((1-\delta)^{\frac{n+3}{2}}-1\right)\left(\frac{1}{2}+\gamma\right)\frac{2}{n+3} A
.
\end{split}
\end{align}
Finally, from \eqref{eq:asympt_n}, \eqref{eq:davon_wollen_wir_asy} and \eqref{eq:asympt_trace}, we obtain \eqref{eq:goal}.
\end{proof}
\section{Proofs of Theorem \ref{th:main} and Theorem \ref{th:test_fct}}\label{s:th_proofs}
In this section, we prove Theorem \ref{th:main} using Proposition \ref{prop:n_tr}. Then we explain how to prove Theorem \ref{th:test_fct} using Theorem \ref{th:main}. 
\begin{proof}[Proof of Theorem \ref{th:main}]
We define the projection operator
\begin{equation}
\Gamma_\lambda:=\sum_{\lambda_j<\lambda}\ket{\Phi_j}\bra{\Phi_j}.
\end{equation}
Then by the linearity of the trace, we have for all $\delta\in[-1/2,1/2]\setminus\{0\}$ 
\begin{align}
\begin{split}
&\qquad \sum_{\lambda_j<\lambda}\int_{X_{\lambda^\gamma}}\left|\Phi_j\right|^2\, \mathrm{d}\vol_g=\Tr\left(\Gamma_\lambda 1_{X_{\lambda^\gamma}}\right)\\
&=\frac{1}{\delta\lambda}\left(
\Tr\left(\Gamma_\lambda\left(\Delta_g+\delta\lambda 1_{X_{\lambda^\gamma}}-\lambda\right)\right)-
\Tr\left(\Gamma_\lambda\left(\Delta_g-\lambda\right)\right)
\right)\\
&=\frac{1}{\delta\lambda}\left(
\Tr\left(\Gamma_\lambda\left(\Delta_g+\delta\lambda 1_{X_{\lambda^\gamma}}-\lambda\right)\right)-
\Tr\left(\Delta_g-\lambda\right)_-
\right),
\end{split}
\end{align}
where we used in the last step that $\Gamma_\lambda$ is the projection on the negative spectral subspace of $\Delta_g-\lambda$. By the variational principle, we have
\begin{equation}
\Tr\left(\Gamma_\lambda\left(\Delta_g+\delta\lambda 1_{X_{\lambda^\gamma}}-\lambda\right)\right)\geqslant \Tr\left(\Delta_g+\delta\lambda 1_{X_{\lambda^\gamma}}-\lambda\right)_-.
\end{equation}
Thus, for $0<\delta<1/2$, we get
\begin{align}\label{eq:pos_delta_phi}
\begin{split}
\sum_{\lambda_j<\lambda}\int_{X_{\lambda^\gamma}}\left|\Phi_j\right|^2\, \mathrm{d}\vol_g&\geqslant\frac{1}{\delta\lambda}\left(
\Tr\left(\Delta_g+\delta\lambda 1_{X_{\lambda^\gamma}}-\lambda\right)_- -
\Tr\left(\Delta_g-\lambda\right)_-\right)\\
&= \frac{1}{\delta\lambda}\left(
\left|\Tr\left(\Delta_g-\lambda\right)_-\right|-\left|\Tr\left(\Delta_g+\delta\lambda 1_{X_{\lambda^\gamma}}-\lambda\right)_-\right|
\right).
\end{split}
\end{align}
and similarly, for $-1/2<\delta<0$, we have
\begin{equation}\label{eq:neg_delta_phi}
\lim_{\lambda\to\infty}\frac{1}{N(\lambda)}\sum_{\lambda_j<\lambda}\int_{X_{\lambda^\gamma}}\left|\Phi_j\right|^2\, \mathrm{d}\vol_g\leqslant
\frac{1}{\delta\lambda}\left(
\left|\Tr\left(\Delta_g-\lambda\right)_-\right|-\left|\Tr\left(\Delta_g+\delta\lambda 1_{X_{\lambda^\gamma}}-\lambda\right)_-\right|\right).
\end{equation}

\bigskip

Now, using \eqref{eq:pos_delta_phi} and Proposition \ref{prop:n_tr}, we get for all $0<\delta<1/2$
\begin{align}
\begin{split}
&\qquad\liminf_{\lambda\to\infty}\frac{1}{N(\lambda)}\sum_{\lambda_j<\lambda}\int_{X_{\lambda^\gamma}}\left|\Phi_j\right|^2\, \mathrm{d}\vol_g\\
&
\geqslant \liminf_{\lambda\to\infty}\frac{1}{N(\lambda)}
 \frac{1}{\delta\lambda}\left(
\left|\Tr\left(\Delta_g-\lambda\right)_-\right|-\left|\Tr\left(\Delta_g+\delta\lambda 1_{X_{\lambda^\gamma}}-\lambda\right)_-\right|
\right)\\
&= \frac{1}{\delta}\frac{2}{n+3}\left(1-(1-\delta)^{\frac{n+3}{2}}\right) 2(1 / 2+\gamma)
.
\end{split}
\end{align}
In particular, since the left-hand side does not depend on $\delta$, we can let $\delta\downarrow0$ and obtain
\begin{align}\label{eq:liminf_fin}
\begin{split}
&\qquad\liminf_{\lambda\to\infty}\frac{1}{N(\lambda)}\sum_{\lambda_j<\lambda}\int_{X_{\lambda^\gamma}}\left|\Phi_j\right|^2\, \mathrm{d}\vol_g\\
&\geqslant\lim_{\delta\downarrow0}\frac{1}{\delta}\frac{2}{n+3}\left(1-(1-\delta)^{\frac{n+3}{2}}\right) 2(1 / 2+\gamma)=2(1 / 2+\gamma).
\end{split}
\end{align}
Similarly, by \eqref{eq:neg_delta_phi}, Proposition \ref{prop:n_tr} and letting $\delta\uparrow0$, we also have
\begin{equation}\label{eq:limsup_fin}
\limsup_{\lambda\to\infty}\frac{1}{N(\lambda)}\sum_{\lambda_j<\lambda}\int_{X_{\lambda^\gamma}}\left|\Phi_j\right|^2\, \mathrm{d}\vol_g\leqslant 2(1 / 2+\gamma).
\end{equation}
Combining \eqref{eq:liminf_fin} and \eqref{eq:limsup_fin}, we obtain \eqref{eq:th_main}.
\end{proof}
At this point, we would like to remark that the idea of considering differences of traces and letting a small parameter $\delta\downarrow 0$ and $\delta\uparrow 0$, can also be found in \cite{read,larry}.

\bigskip

Next, we prove Theorem \ref{th:test_fct}. 

\begin{proof}[Proof of Theorem \ref{th:test_fct}]
Since $f$ continuous and bounded, we may approximate it in $L^\infty$ on $[-1/2,0]\times M$ by a finite sum of indicator functions on sets that have a product structure and are of the form $(-\infty,\gamma]\times \tilde M$, where $\gamma\leqslant0$ and $\tilde M$ is an open subset of $M$ with a piecewise smooth boundary. Due to linearity of both sides of \eqref{eq:th_test_fct}, it suffices to show \eqref{eq:th_test_fct} for test functions $f(x,y)=1_{(-\infty,\gamma]\times \tilde M}$. In the case $\tilde M=M$, this is precisely the statement of Theorem \ref{th:main}. For the general case, one decomposes $M$ into a union of $\tilde M$ with piecewise smooth boundaries and applies Dirichlet-Neumann bracketing for the variable $y\in M$. The proof for each $\tilde M$ follows precisely as the proof of Theorem \ref{th:main} with the only difference being the Dirichlet or Neumann boundary conditions at the boundary of $\tilde M$.
\end{proof}
\printbibliography
\end{document}